\documentclass[11pt]{amsart}

\usepackage{booktabs}
\usepackage{fullpage}
\usepackage{hyperref}
\hypersetup{
    colorlinks,
    citecolor=blue,
    filecolor=blue,
    linkcolor=blue,
    urlcolor=blue
}
\usepackage[alphabetic,backrefs,lite]{amsrefs}

\usepackage{url,amssymb,amsmath,graphicx,mathrsfs}

\usepackage[all]{xy}

\newtheorem{theorem}{Theorem}
\newtheorem{lemma}[theorem]{Lemma}
\newtheorem{corollary}[theorem]{Corollary}
\newtheorem{proposition}[theorem]{Proposition}

\theoremstyle{definition}
\newtheorem{definition}[theorem]{Definition}
\newtheorem{example}[theorem]{Example}

\newtheorem*{warning}{Warning}

\theoremstyle{remark}
\newtheorem{remark}[theorem]{Remark}

\newcommand{\KS}{K^{\rm s}}
\newcommand{\Ss}{S^{\rm s}}

\newcommand{\Jac}{{\rm Jac}}
\newcommand{\Pic}{{\rm Pic}}

\newcommand{\Hom}{{\rm Hom}}
\newcommand{\Spec}{{\rm Spec}}
\newcommand{\Res}{{\rm Res}}
\newcommand{\Wr}{{\mathfrak{R}}}

\newcommand{\rk}{{\rm rk}}

\title{Finite Weil restriction of curves}

\author{E.V.\ Flynn}
\address{Mathematical Institute, University of Oxford, 24--29 St.\ Giles,
Oxford OX1 3LB, United Kingdom}
\email{flynn@maths.ox.ac.uk}
\author{D.\ Testa}
\address{Mathematics Institute, Zeeman Building, University of Warwick, 
Coventry CV4 7AL, United Kingdom}
\email{D.Testa@warwick.ac.uk}

\thanks{Both authors have been partially
supported by EPSRC grant EP/F060661/1.}
\date{30 September, 2012}
\bibstyle{plain}
\begin{document}

\begin{abstract}
Given number fields $L \supset K$, smooth projective curves~$C$ 
defined over $L$ and~$B$ defined over $K$, and a non-constant 
$L$-morphism $h \colon C \to B_L$,
we consider the curve $C_h$ defined over~$K$ whose $K$-rational
points parametrize the $L$-rational points on~$C$ whose images
under~$h$ are defined over~$K$. Our construction 
gives precise criteria for deciding the applicability of Faltings' Theorem 
and the Chabauty method to find the points of the curve $C_h$.  
We provide a framework
which includes as a special case that used in Elliptic Curve
Chabauty techniques and their higher genus versions.
The set $C_h(K)$ can be infinite only when~$C$ has genus at
most~$1$; we analyze completely the case when~$C$ has genus~1. 
\end{abstract}

\maketitle

\section*{Introduction}

Let $L$ be a number field of degree $d$ over $\mathbb{Q}$, let 
$f \in L[x]$ be a polynomial with coefficients in $L$, and define 
\[
A_f := \bigl\{ x \in L ~ \mid ~ f(x) \in \mathbb{Q} \bigr\} .
\]
Choosing a basis of $L$ over $\mathbb{Q}$ and writing explicitly the 
conditions for an element of $L$ to lie in $A_f$, it is easy to see 
that the set $A_f$ is the set of rational solutions of $d-1$ polynomials 
in $d$ variables.  Thus we expect the set $A_f$ to be the set of rational 
points of a (possibly reducible) curve $C_f$; indeed, this is always 
true when $f$ is non-constant.  A basic question that we would like to 
answer is to find conditions on $L$ and $f$ that guarantee that the 
set $A_f$ is finite, and ideally to decide when standard techniques 
can be applied to explicitly determine this set.

We formalize and generalize the previous problem as follows.  
Let $L \supset K$ be a finite separable field extension, let 
$B \to \Spec(K)$ be a smooth projective curve defined over $K$ and 
let $C$ be a smooth projective curve defined over $L$.  Denote by $B_L$ 
the base-change to $L$ of the curve $B$, and suppose that $h \colon C \to B_L$ 
is a non-constant morphism defined over $L$.  Then there is a (possibly 
singular and reducible) curve $C_h$ defined over $K$ whose $K$-rational 
points parametrize the $L$-rational points $p \in C(L)$ such that 
$h(p) \in B(K) \subset B_L(L)$.  
In this context, Theorem~\ref{cocha} identifies an abelian subvariety $F$ of 
the Jacobian of the curve $C_h$ and provides a formula to compute the 
rank of the Mordell-Weil group of $F$: these are the ingredients needed to 
apply the Chabauty method to determine the rational points of the curve $C_h$.
To see the relationship of this general 
problem with the initial motivating question, we let $C:=\mathbb{P}^1_L$ 
and $B:=\mathbb{P}^1_\mathbb{Q}$.  The polynomial $f$ determines a 
morphism $h \colon \mathbb{P}^1_L \to (\mathbb{P}^1_{\mathbb{Q}})_L$ 
and the $L$-points of $\mathbb{P}^1_L$ (different from $\infty$) with 
image in $\mathbb{P}^1(\mathbb{Q})$ correspond to the set $A_f$.

Over an algebraic closure of the field $L$, the curve $C_h$ is isomorphic 
to the fibered product of morphisms $h_i \colon C_i \to B_L$ obtained 
from the initial morphism $h$ by taking Galois conjugates (Lemma~\ref{profi}).
Thus we generalize further our setup: we concentrate our attention on the 
fibered product of finitely many morphisms 
$h_1 \colon C_1 \to B$, \ldots , $h_n \colon C_n \to B$, 
where $C_1 , \ldots , C_n$ and $B$ are smooth curves and the morphisms 
$h_1,\ldots,h_n$ are finite and separable.  We determine the geometric 
genus of the normalization of 
$C_h$, as well as a natural abelian subvariety $J_h$ of the Jacobian of $C_h$.  
Due to the nature of the problem and of the arguments, it is immediate to 
convert results over the algebraic closure to statements over the 
initial field of definition.

Suppose now that $K$ is a number field. We may use the Chabauty method 
to find the rational points on $C_h$, provided the abelian variety $J_h$ 
satisfies the condition that the 
rank of $J_h(K)$ is less than the genus of~$C_h$;
by Chabauty's Theorem (see \cites{prolegom,chab1,chab2})
this guarantees that $C_h(K)$ is finite. Chabauty's Theorem has
been developed into a practical technique, which has been
applied to a range of Diophantine problems, for example
in \cites{colemanchab,colemanpadic,flynnchab,lortuck}.
Further, if the set $C_h(K)$ is infinite, then, by Faltings' 
Theorem \cite{faltings}, the curve $C_h$ contains a component 
of geometric genus at most one; thus the computation of the geometric 
genus of the normalization of $C_h$ is a first step towards answering 
the question of whether $C_h(K)$ is finite or not. Moreover, since all 
the irreducible components of $C_h$ dominate the curve $C$, it follows 
that the set $C_h(K)$ can be infinite only in the case in which the 
curve $C$ has geometric genus at most one. In the case in which $C$ has 
genus zero, results equivalent to special cases of this question have 
already been studied (\cites{az,bt,pa,sch,za}). We shall analyze completely the 
case in which the genus of $C$ is one and the curve $C_h$ has infinitely 
many rational points. This covers as a special case the method
called Elliptic Curve Chabauty (which is commonly applied 
to an elliptic curve~$E$ defined over a number field~$L \supset K$,
satisfying that the rank of $E(L)$ is less than $[L : K]$ and
we wish to find all $(x,y) \in E(L)$ subject to an arithmetic
condition such as $x \in K$); see, for example, 
\cites{bruinth,bruinchab,flywet1,flywet2,wethth}
and a hyperelliptic version in \cite{siksekchab}).

\bigskip

\begin{example}
Let $K$ be a number field and let
$E : y^2 = (a_2 x^2 + a_1 x + a_0)(x + b_1 + b_2 \sqrt{d})$,
with $a_0,a_1,a_2,b_1,b_2,d \in K$,
be an elliptic curve defined over $L = K(\sqrt{d})$; suppose
also that $b_2 \not= 0$, $d \not= 0$, $d \not\in (K^*)^2$,
so that $E$ is not defined over~$K$. 
We are interested in $(x,y) \in  E(L)$ with $x \in K$.
Let $y = r + s\sqrt{d}$, with $r,s \in K$. Equating coefficients
of $1,\sqrt{d}$ gives
\[ 
r^2 + d s^2 = (a_2 x^2 + a_1 x + a_0) (x + b_1),\ \
2 r s = (a_2 x^2 + a_1 x + a_0) b_2.
\]
Let $t = s^2/(a_2 x^2 + a_1 x + a_0)$.
Eliminate~$r$ to obtain $b_2^2/(4 t) + d t = x + b_1$,
so that $x = x(t) := b_2^2/(4 t) + d t - b_1$.
Hence $s,t$ satisfy the curve $C : (st)^2 = t^3(a_2 x(t)^2 + a_1 x(t) + a_0)$,
for which the right hand side is a quintic in~$t$.
One can check directly that this quintic has discriminant
whose factors are powers of: $a_0, b_2, d$, the discriminant
of~$E$ and its conjugate; all of these are guaranteed to be
nonzero, from our assumptions, so that~$C$ has genus~$2$.
By Faltings' Theorem, $C(K)$ is finite, so there are finitely
many such $s,t\in K$, and hence finitely many $x = x(t)$,
and so finitely many $(x,y) \in  E(L)$ with $x \in K$.
Furthermore, one can check directly from the induced map from $C$ to $E$
that, if $J$ is the Jacobian of~$C$ then 
$\hbox{rank}\bigl(J(K)\bigr) = \hbox{rank}\bigl(E(L)\bigr)$; 
hence if $\hbox{rank}\bigl(E(L)\bigr) < [L:K] = 2$
(the condition for Elliptic Curve Chabauty)
then $C$ satisfies the conditions for Chabauty's Theorem over~$K$.
So, for this special case, there is a perfect match between both conditions.
\end{example}

\begin{example}
Let $f \in \mathbb{Q}(\sqrt{2})[x]$ be the polynomial 
$f(x) = x^2 (x - \sqrt{2})$ and suppose that we are interested in 
finding the set $A_f$ of values of $x \in \mathbb{Q}(\sqrt{2})$ such that 
$f(x) \in \mathbb{Q}$.  Writing $x=a+b \sqrt{2}$ with $a,b \in \mathbb{Q}$ 
and substituting in $f$ we find that $f(a + b \sqrt{2})$ is a rational 
number if and only if the equality $a^2+2b^2=3a^2b+2b^3$ holds.  Apart 
from the solution $(a,b)=(0,0)$, all the remaining solutions of the 
resulting equation can be determined by the substitution $a=tb$ and 
the set $A_f$ is the set 
\[
\left\{ \frac{t^2+2}{(3t^2+2)} (t+\sqrt{2})~ \Bigl| ~ t 
\in \mathbb{Q} \right\} \cup \bigl\{ 0 \bigr\} .
\]
\end{example}
\begin{example}
Let $f \in \mathbb{Q}(\sqrt{2})[x]$ be the polynomial 
$f(x) = \frac{x (x - \sqrt{2})}{x-1}$.
Arguing similarly to the previous example, we find that the set values of 
$x \in \mathbb{Q}(\sqrt{2})$ such that $f(x) \in \mathbb{Q}$ consists of 
the solutions to the equation $a^2 b - a^2 - 2 a b + a - 2 b^3 + 2 b^2 = 0$.  
The projective closure of the previous equation defines an elliptic 
curve $E$ with Weierstrass form $y^2=x^3-x$; by a 2-descent it is easy 
to show that $E(\mathbb{Q}) = E(\mathbb{Q})[2]$ and we conclude that 
the set we are seeking is the set $\{ 0 , \sqrt{2} \}$.
\end{example}

The last two of these examples exhibit the two qualitative behaviours 
that we analyze in what follows.

\section{Relative Weil Restriction} \label{rwr}

We begin this section by recalling the definition of the Weil restriction 
functor.  The setup is quite general, though we will only use it in a 
very specialized context.  We prefer to adopt this formal point of view at 
the beginning, since it simplifies the arguments; for the cases mentioned 
in the introduction, it is straightforward to translate all our arguments 
into explicit computations that are also easy to verify.

Let $s \colon S' \to S$ be a morphism of schemes and let $X'$ be 
an $S'$-scheme; the contravariant functor 
\[
\begin{array}{rcl}
\Wr_{S'/S}(X') \colon ({\rm Sch}/S) ^o & \longrightarrow & {\rm Sets} \\[5pt]
T & \longmapsto & {\rm Hom}_{S'} (T \times _S S' , X')
\end{array}
\]
is the {\em Weil restriction functor}.  If the functor $\Wr_{S'/S}(X')$ 
is representable, then we denote by $\Wr_{S'/S}(X')$ also the scheme 
representing $\Wr_{S'/S}(X')$; sometimes, to simplify the notation, we 
omit the reference to $S$ and $S'$ and write that $\Wr(X')$ is the Weil 
restriction of $X'$.  The scheme $\Wr(X')$ is determined by the isomorphism
 
\[
{\rm Hom}_S ( - , \Wr(X')) \stackrel{\sim}{\longrightarrow} {\rm Hom}_{S'} 
(- \times_S S', X') 
\]
of functors $({\rm Sch}/S)^o \to {\rm Sets}$.  Informally this means that 
the $S$-valued points of $\Wr(X')$ are the same as the $S'$-valued points 
of $X'$.

Suppose now that $Y \to S$ is another morphism of schemes, denote by $Y'$ 
the fibered product $Y \times _S S'$ with natural morphism 
$b \colon Y' \to Y$, and let $h \colon X \to Y'$ be any scheme.  
We have the diagram 
\begin{equation} \label{diato}
\begin{minipage}{150pt}
\xymatrix{
X \ar[d] _h \\
Y' \ar[d] \ar[r]^b & Y \ar[d] \\
S' \ar[r] ^{s} & S }
\end{minipage}
\end{equation}
and $X$ is therefore both a $Y'$-scheme and an $S'$ scheme.  Thus 
there are two possible Weil restrictions we can construct: 
\begin{itemize}
\item
the Weil restriction $\Wr_{S'/S} (X)$, using the $S'$-scheme structure of $X$, 
\item
the Weil restriction $\Wr_{Y'/Y} (X)$, using the $Y'$-scheme structure of $X$.
\end{itemize}
Below we shall use the notation $\Res_h(X)$ for the Weil restriction 
$\Wr_{Y'/Y}(X)$ and call it the {\em{relative Weil restriction}}.

In order to relate the Weil restriction $\Wr_{Y'/Y}(X)$ to the discussion 
in the introduction, we give an alternative definition of this functor and 
then proceed to prove the equivalence of the two.  For concreteness, suppose 
that in diagram~\eqref{diato} the morphism $S' \to S$ is induced by a 
(finite, separable) field extension $L \supset K$; then there is a subset 
of the $L$-points of $X$ whose image under $h$ is not simply an $L$-point 
of $Y'$, but it is actually a $K$-point of $Y$.  We would like to say that 
this set of $L$-points of $X$ with $K$-rational image under $h$ are the 
$K$-rational points of a scheme defined over $K$, and that this scheme 
{\em{is}} the relative Weil restriction of $X$.  This is what we discussed 
in the introduction and we now formalize it.

Hence, let $T$ be any $S$-scheme and denote by $T'$ the $S'$-scheme 
$T \times _S S'$.  Pull-back by the morphism $s$ defines a function 
${\rm Hom}_S (T,Y) \to {\rm Hom}_{S'} (T',Y')$; there is also a function 
${\rm Hom}_{S'} (T',X) \to {\rm Hom}_{S'} (T',Y')$ determined by composition 
with $h$.  Summing up, for any $S$-scheme $T$, we obtain a diagram 
\begin{equation} \label{sta}
\begin{minipage}{150pt}
\xymatrix{{\rm Hom}(T,X/Y) \ar@{-->}[r] \ar@{-->}[d] 
 & {\rm Hom}_{S'} (T',X) \ar[d] ^{h \circ -}\\
{\rm Hom}_S (T,Y) \ar[r] ^{s^*} & {\rm Hom}_{S'} (T',Y') .
}
\end{minipage}
\end{equation}
We denote by ${\rm Hom}(T,X/Y)$ the pull-back of diagram~\eqref{sta}, and 
we define the {\em relative Weil restriction functor} to be the functor 
\[
\begin{array}{rcl}
\Res_h \colon ({\rm Sch}/S)^o & \longrightarrow & {\rm Sets} \\[5pt]
T & \longmapsto & {\rm Hom}(T,X/Y). 
\end{array}
\]
If the functor $\Res_h$ is representable, we denote a scheme representing 
it by $\Res_h(X)$.

\begin{lemma} \label{lem:rwr}
Let $s \colon S' \to S$ be a morphism of schemes.  Let $X$ be an $S'$-scheme 
and let $Y$ be an $S$-scheme; denote by $Y_{S'}$ the $S'$-scheme 
$Y \times _S S'$.  Let $h \colon X \to Y_{S'}$ be an $S'$-morphism and 
assume that both Weil restriction functors $\Wr(X)$ and $\Wr(Y_{S'})$ are 
representable (notation as in~\eqref{diato}).  Then the relative Weil 
restriction functor $\Res_h \colon ({\rm Sch}/S)^o \to {\rm Sets}$ is 
representable and the schemes $\Res_h(X)$ and $\Wr_{Y'/Y}(X)$ coincide.  
Moreover, there is a commutative diagram of $S$-schemes 
\begin{equation} \label{dere}
\begin{minipage}{150pt}
\xymatrix{
\Res_h(X) \ar[r] \ar[d] & \Wr(X) \ar[d]^{h'} \\
Y \ar[r]^{\iota~~~} & \Wr(Y_{S'}) 
}
\end{minipage}
\end{equation}
exhibiting $\Res_h(X)$ as a fibered product of $\Wr(X)$ and $Y$ over 
$\Wr(Y_{S'})$.
\end{lemma}

\begin{proof}
We begin by constructing an $S$-morphism $h' \colon \Wr(X) \to \Wr(Y_{S'})$.  
By representability of $\Wr(Y_{S'})$ and of $\Wr(X)$ there are natural 
bijections 
\[
\begin{array}{rcll}
{\rm Hom}_{S} (\Wr(X) , \Wr(Y_{S'})) & \stackrel{\sim}{\longrightarrow} 
& {\rm Hom}_{S'} (\Wr(X)_{S'} , Y_{S'}) 
& \quad
{\textrm{and}} \\[7pt]
{\rm Hom}_{S} (\Wr(X) , \Wr(X)) & \stackrel{\sim}{\longrightarrow} 
& {\rm Hom}_{S'} (\Wr(X)_{S'} , X) .
\end{array}
\]
Let $\gamma \in {\rm Hom}_{S'} (\Wr(X)_{S'} , X)$ be the $S'$-morphism 
corresponding to the identity in ${\rm Hom}_{S} (\Wr(X) , \Wr(X))$ and 
let $h' \colon \Wr(X) \to \Wr(Y_{S'})$ be the $S$-morphism corresponding 
to $h \circ \gamma \in {\rm Hom}_{S'} (\Wr(X)_{S'} , Y_{S'})$.  Note also 
that there is an $S$-morphism $\iota \colon Y \to \Wr(Y_{S'})$ corresponding 
to the identity in ${\rm Hom}_{S'} (Y_{S'}, Y_{S'})$.  Define 
$\Res_h(X) := Y \times _{\Wr(Y_{S'})} \Wr(X)$ so that there is a commutative 
diagram as in~\eqref{dere}.  To check that $\Res_h(X)$ represents $\Res_h$, 
let $T$ be any $S$-scheme; we have 
\begin{eqnarray*}
{\rm Hom}_S (T,\Res_h(X)) & = & {\rm Hom}_S (T,Y) 
\times_{{\rm Hom}_S (T , \Wr(Y_{S'}))} {\rm Hom}_S (T,\Wr(X)) = \\[5pt]
& = & {\rm Hom}_S (T,Y) \times_{{\rm Hom}_S (T,\Wr(Y_{S'}))} {\rm Hom}_{S'} 
(T_{S'},X) = {\rm Hom} (T,X/Y)
\end{eqnarray*}
as required.  Thus, the scheme $\Res_h(X)$ defined above represents the 
functor $\Res_h$ and is obtained by the fibered product~\eqref{dere}.

Finally, we check that the schemes $\Res_h(X)$ and $\Wr_{Y'/Y}(X)$ coincide.  
Let $T$ be any $Y$-scheme; we have 
$T \times_Y Y' = T \times _Y (Y \times_S S') 
= (T \times _Y Y) \times_S S' = T \times_S S'$ and 
\begin{eqnarray*}
{\rm Hom}_Y (T,\Wr(X)) & = & {\rm Hom}_{Y'} (T \times_Y Y',X) = {\rm Hom}_{Y'} 
(T \times _S S',X) = \\[5pt]
& = & \Bigl\{ f \in {\rm Hom}_{S'} (T \times_S S',X) ~\Bigl|~ h \circ f 
= s^*(T_Y \to Y) \Bigr\} = \\[5pt]
& = & {\rm Hom} (T,X/Y) = {\rm Hom}_Y (T, \Res_h(X)) .
\end{eqnarray*}
We conclude using Yoneda's Lemma that $\Res_h(X)$ and $\Wr_{Y'/Y}(X)$ 
coincide, and hence the lemma follows.
\end{proof}

Informally, we may describe the scheme $\Res_h(X)$ as the scheme 
whose $S$-valued points correspond to the points in $X(S')$ lying above 
the points in $Y(S)$.

Observe that, by the construction of $\Res_h(X)$ and diagram~\eqref{dere}, 
there is an $S'$-morphism $\Res_h(X)_{S'} \to X$.  The morphism 
$\Res_h(X)_{S'} \to X$ will prove useful later.

We now introduce some notation that is used in the following lemma.  
Let $L \supset K$ be a finite, separable field extension of degree $n$, 
and let $\KS$ be a separable closure of $K$.  Set $S' := {\rm Spec}(L)$, 
$S := {\rm Spec}(K)$, and $\Ss := {\rm Spec}(\KS)$, denote by $s \colon S' 
\to S$ the morphism of schemes corresponding to the extension $L \supset K$ 
and suppose that $\varphi_1 , \ldots , \varphi_n \colon \Ss \to S'$ are the 
morphisms corresponding to all the distinct embeddings of $L$ into $\KS$ 
fixing $K$.  Given an $S'$-scheme $X$, we denote by $X_1 , \ldots , X_n$ 
the $\Ss$-schemes obtained by pulling back $X \to S'$ under the morphisms 
$\varphi_1 , \ldots , \varphi_n$.

\begin{lemma} \label{profi}
Let $X$ be an $S'$-scheme and let $Y$ be an $S$-scheme.  
Let $h \colon X \to Y_{S'}$ be an $S'$-morphism and assume that both Weil 
restriction functors $\Wr(X)$ and $\Wr(Y_{S'})$ are representable.  
Then the isomorphism 
\begin{equation} \label{diago}
\Res_h (X)_{{\Ss}} \simeq X_1 \times_{Y_{{\Ss}}} X_2 \times_{Y_{{\Ss}}} 
\cdots \times_{Y_{{\Ss}}} X_n
\end{equation}
holds.  More precisely, the isomorphism~\eqref{diago} holds replacing $\KS$ 
by any normal field extension of $K$ containing $L$.
\end{lemma}

\begin{proof}
To prove the result, it suffices to consider the case in which both $X$ 
and $Y$ are affine; thus suppose that $Y = {\rm Spec}(A)$, 
$X = {\rm Spec}(B)$, and that  $\underline{x}$ denotes a set of generators 
of $B$ as an $A$-algebra.

Let $\alpha_1 , \ldots , \alpha_n$ denote a basis of $K$ over $L$ and 
let $\underline{x}_1 , \ldots , \underline{x}_n$ denote disjoint sets of 
variables each in bijection with $\underline{x}$; whenever we denote a 
variable in $\underline{x}$ by a symbol such as $x$, we denote by the 
symbols $x_1 , \ldots x_n$ the variables corresponding to $x$ 
in $\underline{x}_1 , \ldots , \underline{x}_n$.  
For every $x \in \underline{x}$ 
define linear forms $\tilde{x}_1 := \sum _i x_i \varphi_1(\alpha_i) , 
\ldots , \tilde{x}_n := \sum _i x_i \varphi_n(\alpha_i)$ in 
$A[\underline{x}_1 , \ldots , \underline{x}_n]$.  First we show that the 
forms $\tilde{x}_1 , \ldots , \tilde{x}_n$ are linearly independent.  
Indeed, let $\Delta$ be the matrix whose $(i,j)$ entry is 
$\varphi_j(\alpha_i)$; the $(i,j)$ entry of $\Delta \Delta^t$ 
is $\sum _k \varphi_k(\alpha_i \alpha_j) = {\rm Tr}_{L/K}(\alpha_i \alpha_j)$.
Thus the determinant of $\Delta \Delta ^t$ is equal to the discriminant 
of $L$ over $K$, and it is therefore non-zero.  We deduce that the 
matrix $\Delta$ is invertible and that the forms 
$\tilde{x}_1 , \ldots , \tilde{x}_n$ are independent.  Thus, defining 
$\tilde{\underline{x}}_j := \{ \tilde{x}_j \mid x \in \underline{x} \}$ 
for $j \in \{1,\ldots,n\}$, we proved that there is an isomorphism 
$\KS \otimes_K A[\underline{x}_1 , \ldots , \underline{x}_n] 
\simeq \KS \otimes_K A[\tilde{\underline{x}}_1 , \ldots , 
\tilde{\underline{x}}_n]$.

The relative Weil restriction of $X$ may be defined as follows.  
Let $g(\underline{x})$ be an element of the polynomial ring 
$A[\underline{x}]$ contained in the ideal defining $X$; evaluate 
$g(\underline{x})$ substituting for each variable $x \in \underline{x}$ 
the sum $\sum x_i \alpha_i$ and write the resulting polynomial in 
$L \otimes_K A[\underline{x}_1 , \ldots , \underline{x}_n]$ as 
$\sum g_i \alpha_i$, where $g_1 , \ldots , g_n$ are elements of 
$A[\underline{x}_1 , \ldots , \underline{x}_n]$; denote the sequence 
$(g_1 , \ldots , g_n)$ by $\tilde{g}$.  Then the scheme $\Res(X)$ is the 
scheme in ${\rm Spec}(A[\underline{x}_1 , \ldots , \underline{x}_n])$ whose 
ideal is the ideal generated by the elements of $\tilde{g}$, as $g$ varies 
among all the elements of the ideal defining $X$.  It is therefore clear 
that the ideal $I$ defining $\Res(X)$ in $\KS \otimes_K 
A[\underline{x}_1 , \ldots , \underline{x}_n]$ contains, for every 
embedding $\varphi \colon L \to \KS$ fixing $K$, the elements 
$\sum_i g_i \varphi(\alpha_i) = \varphi \bigl( \sum g_i \alpha_i) 
= \varphi(g)$, and conversely that the ideal containing all such elements 
contains $g_1 , \ldots , g_n$ and hence it contains $I$.  Let 
$\mathscr{F} \in A[\underline{x}]$ be a set of generators of the ideal 
of $X$; since, for $i \in \{1,\ldots , n\}$, the scheme $X_i$ is defined 
by $\{\varphi_i(f) \mid f \in \mathscr{F} \}$ in $A[\tilde{\underline{x}}_i]$ 
the result follows.
\end{proof}

\section{The case of curves over an algebraically closed field}

In this section we compute the geometric genus of the relative Weil 
restriction $\mathscr{C}$ of a curve and we also determine a natural 
abelian variety isogenous to a subvariety of the Jacobian of $\mathscr{C}$.  
To calculate the geometric genus in Theorem~\ref{thm:fibre} we may clearly 
assume that the ground field is algebraically closed; the proof reduces 
the computation to the \'etale local case, settled in Lemma~\ref{lem:fibre}.  

\begin{lemma} \label{lem:fibre}
Let $n$ be a positive integer.  Suppose that $k$ is an algebraically 
closed field and that $r_1 , \ldots , r_n$ are positive integers 
relatively prime with the characteristic of $k$.  Let $R$ be the least 
common multiple of $r_1 , \ldots , r_n$ and let 
$C \subset \mathbb{A}_{x,y_1,\ldots,y_n}^{n+1}$ be the affine scheme 
defined by 
\[
C \colon \left\{
\begin{array}{rcl}
y_1^{r_1} & = & x , \\
& \vdots \\
y_n^{r_n} & = & x .
\end{array}
\right.
\]
The scheme $C$ has $\frac{r_1 \cdots r_n}{R}$ irreducible components 
and the morphism induced by the projection onto the $x$-axis from the 
normalization of each component ramifies at the origin to order $R$.
\end{lemma}

\begin{proof}
Observe that the curve $C$ carries the action $\rho$ of $\mathbb{G}_m$ 
defined by 
\[ 
t \cdot (x,y_1,\ldots,y_n) = (t^R x , t^{R/r_1} y_1 , 
\ldots , t^{R/r_n} y_n).  
\]
First we show that the action $\rho$ has trivial 
stabilizer at all points of $C$ different from the origin.  Indeed, let 
$q=(x,y_1,\ldots,y_n)$ be a point of $C$ different from the origin; it 
follows that all the coordinates of $q$ are non-zero and the equality 
$(x,y_1,\ldots,y_n) = (t^R x , t^{R/r_1} y_1 , \ldots , t^{R/r_n} y_n)$ 
implies that $t$ is a root of unity of order dividing 
$\gcd (R/r_1 , \ldots , R/r_n) = 1$.  Thus the complement of the origin 
in $C$ is a principal homogeneous space for $\mathbb{G}_m$.  There 
are $r_1 \cdots r_n$ points on $C$ such that $x=1$ and these are stabilized 
precisely by the action of the $R$-th roots of unity; we deduce that $C$ 
consists of $\frac{r_1 \cdots r_n}{R}$ irreducible components, each 
isomorphic to closure of the orbit of a point $(1,\eta_1 , \ldots , \eta_n)$, 
where $\eta_i$ is an $r_i$-th root of unity,  for $i \in \{1,\ldots,n\}$.  
Therefore the normalization of $C$ consists of $\frac{r_1 \cdots r_n}{R}$ 
components each mapping to the $x$-axis by 
$(t^R,\eta_1 t^{R/r_1}, \ldots , \eta_n t^{R/r_n}) \mapsto t^R$, as required.
\end{proof}

\begin{theorem} \label{thm:fibre}
Let $k$ be an algebraically closed field and let $n$ be a positive integer.  
Suppose that $S,C_1,\ldots,C_n$ are smooth curves; for 
$i \in \{1,\ldots,n\}$ let $f_i \colon C_i \to S$ be a finite separable 
morphism whose ramification indices are coprime to the characteristic of $k$.  
Let $f$ denote the morphism of the curve 
$C := C_1 \times _S \cdots \times _S C_n$ to $S$, and let $C'$ be the 
normalization of $C$; denote by $g_S$ and $g_{C'}$ the arithmetic genera 
of $S$ and $C'$ respectively.  For any point $p \in C$ and 
any $i \in \{1,\ldots,n\}$ denote by $r_i=r_i(p)$ the ramification index 
of $f_i$ at the point corresponding to $p$ and 
let $R=R(p) := {\rm lcm} \{r_1(p) , \ldots , r_n(p) \}$.  Then the 
curve $C$ is a local complete intersection 
and the curve $C'$ has arithmetic genus
\begin{eqnarray*}
g_{C'} & = & 1 + (g_S-1) \prod_{i=1}^n \deg(f_i) + \frac{1}{2} 
\sum _{p \in C} r_1(p) \cdots r_n(p) \left( 1 - \frac{1}{R(p)} \right) = 
\\[5pt]
& = & 1 + \frac{1}{2} (r - 2g_S - 2) \prod_{i=1}^n \deg(f_i) - \frac{1}{2} 
\sum _{p \in f^{-1} (R_{f})} \frac{r_1(p) \cdots r_n(p)}{R(p)} 
\end{eqnarray*}
where $R_f \subset S$ is the union of the sets of branch points of the 
morphisms $f_1 , \ldots , f_n$ and $r$ is the cardinality of $R_f$.
\end{theorem}

\begin{proof}
Let $\pi \colon C' \to S$ be the morphism induced by the structure 
morphism $f \colon C \to S$.  Let $p' \in C'$ be a closed point, let $p$ 
be the corresponding point of $C$ and let $p_1 , \ldots , p_n$ be the 
points of $C_1 , \ldots , C_n$ respectively corresponding to $p$.  
We prove the result by finding a local model of $C$ near $p$ which 
is a local complete intersection, and then applying the Hurwitz formula 
to the morphism induced by $f$ on the normalization of such a model.  
Choosing a local coordinate $x$ on $S$ near $\pi(p')$ we reduce to the 
case in which $S$ is $\mathbb{A}^1$ and $\pi(p') = 0$.  Similarly choose 
local coordinates $z_1 , \ldots , z_n$ on $C_1 , \ldots , C_n$ 
near $p_1 , \ldots , p_n$ respectively.  Thus, near $p$, the curve $C$ 
is defined by 
\[
C \colon \left\{
\begin{array}{rcl}
z_1^{r_1} \varphi_1(z_1) & = & x , \\
& \vdots \\
z_n^{r_n} \varphi_n(z_n) & = & x ,
\end{array}
\right.
\]
where $(\varphi_1 , \ldots , \varphi_n)$ is a rational function on $C$ 
defined and non-zero at $p$, and $r_1 , \ldots , r_n$ are the local 
ramification indices.  In particular, the curve $C$ is a local complete 
intersection near $p$.  Denote by $\mathcal{O}_{p}$ the local ring 
of $C$ near $p$; the base-change defined by the inclusion 
\[
\mathcal{O}_{p} \longrightarrow \mathcal{O}_{p} [t_1 , \ldots , t_n] / 
\bigl( t_1^{r_1} - \varphi_1(z_1) , \ldots , t_n^{r_n} - \varphi_n(z_n) \bigr)
\]
is finite \'etale (of degree $r_1 \cdots r_n$) by the assumption that 
the ramification indices are relatively prime to the characteristic of 
the field $k$.  Hence, each component of the resulting curve is locally 
isomorphic to the curve $C_{p} \subset \mathbb{A}^{n+1}_{x,y_1,\ldots,y_n}$ 
defined by 
\[
C_{p} \colon \left\{
\begin{array}{rcl}
y_1^{r_1} & = & x , \\
& \vdots \\
y_n^{r_n} & = & x ,
\end{array}
\right.
\]
where $y_1 = z_1 t_1 , \ldots , y_n = z_n t_n$.  The morphism induced 
by $\pi$ on $C_{p}$ is the morphism induced by the coordinate $x$.  
Using Lemma~\ref{lem:fibre}, we conclude that the contribution of the 
point $p$ to the Hurwitz formula is 
$\frac{r_1 \cdots r_n}{R} (R - 1) = r_1 \cdots r_n (1 - \frac{1}{R})$, 
and hence we obtain 
\[
2g_C-2 = (2g_S-2) \prod \deg(f_i) + \sum _{p \in C} r_1(p) \cdots r_n(p) 
\left( 1 - \frac{1}{R(p)} \right)
\]
and the first formula follows.  To prove the second one note that the 
quantity $r_1(p) \cdots r_n(p) ({1 - \frac{1}{R(p)}})$ vanishes 
for $p \notin R_f$, since in this case all the local ramification indices 
equal $1$, and that for all points $s$ of $S$ we have 
\[
\sum _{p \in f^{-1}(s)} r_1(p) \cdots r_n(p) 
= \prod _{i=1}^n \sum _{p \in f_i^{-1}(s)} r_i(p) = \prod \deg(f_i) 
\]
and we conclude.
\end{proof}

In the next results, for a projective scheme $Y$, we denote by $\Jac(Y)$ 
the {\em{Jacobian variety of $Y$}}, that is the connected component 
of the identity of the group $\Pic(Y)$.

\begin{lemma} \label{prodo}
Suppose that $X_1 , \ldots , X_n$ are smooth projective varieties defined 
over an algebraically closed field $k$ and let 
$X=X_1 \times \cdots \times X_n$.  For $i \in \{1,\ldots,n\}$ 
let $\rho_i \colon X \to X_i$ denote the canonical projection. The morphism 
\[
\rho_\bullet^* = (\rho_1^* , \ldots , \rho_n^*) \colon \Jac(X_1) 
\times \cdots \times \Jac(X_n) \to \Jac(X)
\]
is an isomorphism.
\end{lemma}

\begin{proof}
Choosing a point in each variety $X_1 , \ldots , X_n$ allows us to define, 
for $i \in \{1,\ldots,n\}$, an inclusion $X_i \hookrightarrow X$.  These 
inclusions in turn determine a section $\Jac(X) \to \Jac(X_1) \times 
\cdots \times \Jac(X_n)$ of the morphism $\rho_\bullet^*$.  We deduce 
that $\rho_\bullet^*$ is indeed an isomorphism of Jacobian varieties 
$\Jac(X_1) \times \cdots \times \Jac(X_n) \simeq \Jac(X)$.
\end{proof}

\begin{remark}
Maintaining the notation of the previous lemma, the morphisms 
$\rho_1,\ldots,\rho_n$ also induce a homomorphism $\psi^* \colon \Pic(X_1) 
\times \cdots \times \Pic(X_n) \to \Pic(X)$.  The morphism $\psi$ though 
need not be an isomorphism.  Since $\psi$ induces an isomorphism on the 
connected component of the identity, it factors through the N\'eron-Severi 
group and in particular its cokernel is a finitely generated abelian group.  
For instance, if $E$ is an elliptic curve and $X_1=X_2=E$, then the three 
classes $\{0\} \times E$, $E \times \{0\}$ and the diagonal are independent 
in ${\rm NS}(E \times E)$ so that ${\rm NS}(E \times E) \supset \mathbb{Z}^3$.
On the other hand, the group ${\rm NS}(E) \times {\rm NS}(E)$ is isomorphic 
to $\mathbb{Z}^2$.
\end{remark}

For the next theorem we need to introduce some notation.  Suppose that 
$B,X_1 , \ldots , X_n$ are smooth projective varieties defined over an 
algebraically closed field $k$.  For each $i \in \{1,\ldots,n\}$ 
let $f_i \colon X_i \to B$ be a finite morphism.  Denote by $X$ the product 
$X_1 \times \cdots \times X_n$, by $X_B$ the fibered product 
$X_1 \times _B \cdots \times _B X_n$ and by $\iota \colon X_B \to X$ the 
natural inclusion.  For each $i \in \{1,\ldots,n\}$ 
\begin{itemize}
\item 
let $d_i$ denote the degree of $f_i$, 
\item 
let $\rho_i \colon X \to X_i$ and $\pi_i=\rho_i \circ \iota \colon X_B 
\to X_i$ be the canonical projections, 
\item 
let $\pi \colon X \to B$ be the composition $f_1 \circ \pi_1 = \cdots 
= f_n \circ \pi_n$, 
\item 
let $d = d_1 \cdots d_n$ denote the degree of $\pi$.
\end{itemize}
We summarize the notation in the case $n=2$ in~\eqref{diauno}.
\begin{equation} \label{diauno}
\begin{array}{c}
X := X_1 \times X_2 \\[25pt]
X_B := X_1 \times _B X_2 \\
\vphantom{\begin{array}{c}W\\W\\W\\W\\W\end{array}}
\end{array}
\hspace{20pt}
\begin{minipage}{150pt}
\xymatrix{
 & X_1 \times X_2 \ar[dddl] _{\rho_1} \ar[dddr] ^{\rho_2} \\ \\
 & {\hphantom{{}_B} X_1} \times_B X_2 \ar[uu] _\iota \ar[dl] ^{\pi_1} 
\ar[dd] ^{\pi} 
 \ar[dr] _{\pi_2} \\
X_1 \ar[dr] ^{f_1} & & X_2 \ar[dl] _{f_2} \\
 & B
}
\end{minipage}
\hphantom{\hspace{20pt}C}
\end{equation}

\begin{warning}
Even though the diagram may suggest it, the identities $f_1 \circ \rho_1 
= \cdots = f_n \circ \rho_n$ {\em{do not}} hold necessarily.  On the other 
hand, the identities $f_1 \circ \pi_1 = \cdots = f_n \circ \pi_n = \pi$ hold.
\end{warning}

Let $\phi \colon \Jac(X) \to \Jac(X)$ be the isogeny defined by 
\[
\begin{array}{ccl}
\phi \colon \quad \Jac(X) & \longrightarrow & \Jac(X) \\[5pt]
\rho_1^*D_1 + \cdots + \rho_n^*D_n & \longmapsto & \frac{d}{d_1}D_1 
+ \cdots + \frac{d}{d_n}D_n ,
\end{array}
\]
where we used the identification of $\Jac(X)$ with $\Jac(X_1) 
\times \cdots \times \Jac(X_n)$ of Lemma~\ref{prodo}.

Let $M'$ denote the kernel of the multiplication map $\Jac(B)^n \to \Jac(B)$, 
so that $M' \simeq \Jac(B)^{n-1}$; define the group $M$ as the image 
of $M'$ under the morphism 
\[
\begin{array}{rcl}
\Jac(B)^n & \longrightarrow & \Jac(X) \\[7pt]
(D_1,\ldots,D_{n-1}) & \longmapsto & \rho_1^*f_1^*D_1 + \cdots 
+ \rho_n^*f_n^*D_n.
\end{array}
\]
By construction, the group $M$ is connected and contained in the kernel 
of $\iota^*$.  Moreover, it follows from Lemma~\ref{prodo} and the 
fact that the morphisms $f_1,\ldots,f_n$ are finite that the morphism 
$M' \to M$ is finite and hence that the dimension of the group $M$ 
is $(n-1) \dim (\Jac(B))$.

\begin{theorem} \label{thm:jaco}
Maintaining the notation introduced above, the group $M$ has finite 
index in $\ker(\iota^*)$.  More precisely, for each element $D$ 
of $\ker(\iota^*)$ there are elements $D_1,\ldots,D_n$ of $\Jac(B)$ 
such that the identities 
\[
\begin{array}{rcll}
\phi(D) & = & \sum_i \rho_i^*f_i^*D_i & \quad {\textrm{and}} \\[7pt]
d \sum_i D_i & = & 0
\end{array}
\]
hold.  In particular, the dimension of the kernel of $\iota^*$ is 
$(n-1) \dim (\Jac(B))$.
\end{theorem}

\begin{proof}
By Lemma~\ref{prodo}, the morphism 
\[
\begin{array}{ccl}
\Jac(X_1) \times \cdots \times \Jac(X_n) & \longrightarrow & \Jac(X) \\[5pt]
(D_1 , \ldots , D_n) & \longmapsto & \rho_1^* D_1 + \cdots + \rho_n^* D_n
\end{array}
\]
is an isomorphism.  Thus we identify the divisor classes in $\Jac(X)$ 
with $n$-tuples of divisor classes, one in each Jacobian variety 
$\Jac(X_1), \ldots , \Jac(X_n)$.  Let $i,j$ be distinct indices in 
$\{1,\ldots,n\}$ and let $P$ be a divisor on $X_i$; it is easy to check 
that the identities 
\begin{eqnarray*}
\pi_i{}_*\pi_i^*(P) & = & \frac{d}{d_i} P \, , 
\\[5pt]
\pi_i{}_*\pi_j^*(P) & = & \frac{d}{d_i d_j} f_i^*f_j{}_*(P) 
\end{eqnarray*}
hold, and hence the class of the divisor $\pi_i{}_*\pi_j^*(P)$ is 
contained in $f_i^* \Pic(B)$.  Suppose that $D=\rho_1^*D_1+\cdots+\rho_n^*D_n$ 
is a divisor on $X$ representing an element of $\Jac(X)$.  Let $i$ be an 
index in $\{1,\ldots,n\}$; we have 
\[
\pi_i{}_* \iota^* (D) = \pi_i{}_* (\pi_1^*D_1+\cdots+\pi_n^*D_n) 
\in \frac{d}{d_i}D_i + f_i^* \Jac(B)
\]
and summing over all indices $i$, we find the equivalence 
\[
\begin{array}{rcl}
\displaystyle \sum _i \pi_i{}_* \iota^* (D) & \equiv & \displaystyle 
\sum_i \frac{d}{d_i}D_i \\[15pt]
& \equiv & \displaystyle \phi(D) \pmod{f_1^*\Jac(B) \times \cdots 
\times f_n^* \Jac(B)} .
\end{array}
\]
In particular, if $D$ is contained in the kernel of $\iota^*$, then 
$\phi(D)$ is contained in $f_1^*\Jac(B) \times \cdots \times f_n^* \Jac(B)$, 
establishing the first of the two identities.  Finally, let 
$D_1 , \ldots , D_n \in \Jac(B)$ be divisor classes such that the 
element $\rho_1^*f_1^*D_1 + \cdots + \rho_n^*f_n^*D_n$ of $\Jac(X)$ lies 
in $\ker (\iota^*)$.  Then, the element $D_1 + \cdots + D_n$ of $\Jac(B)$ 
lies in the kernel of $\pi^*$, so that $d(D_1 + \cdots + D_n) 
= \pi_* \pi^*(D_1 + \cdots + D_n) = 0$, proving the second identity.

It follows from what we just proved that the equalities 
\[
\dim (\ker(\iota^*)) = \dim M = ({n-1}) \dim (\Jac(B))
\]
hold, proving the final assertion of the theorem.
\end{proof}

\section{Mordell-Weil groups and relative Weil restriction}

From now on, we shall be in the following set up (specializing the 
assumptions of Lemma~\ref{lem:rwr}): 
\begin{itemize}
\item 
$L$ is a number field and $S':=\Spec(L)$, 
\item 
$K \subset L$ is a subfield and $S:=\Spec(K)$, 
\item 
$C$ is a smooth projective curve over $L$, 
\item 
$B$ is a smooth projective curve over $K$, and 
\item 
$h \colon C \to B_{S'}=B_L$ is a finite morphism.
\end{itemize}
To simplify the notation, for any variety $Z$ defined over a number 
field $k$ denote by $mw_k(Z)$ the rank of the Mordell-Weil group of 
the Jacobian of $Z$; we are only going to apply this notation with 
$k \in \{K,L\}$ to varieties $Z$ that are either reduced curves or 
products of smooth integral curves.

\begin{theorem} \label{cocha}
Suppose that $C$ is a smooth projective curve defined over a number 
field $L$.  Suppose that $B$ is a smooth projective curve defined 
over a subfield $K$ of $L$, and let $h \colon C \to B_L 
:= B \times_{\Spec(K)} \Spec(L)$ be a finite morphism.  Denote by $n$ 
the dimension of $L$ as a vector space over $K$ and by $g(C)$ and $g(B)$ 
the genera of $C$ and $B$ respectively.  The Jacobian of $\Res_h(C)$ 
contains an abelian subvariety $F$ of dimension $n g(C) - (n - 1) g(B)$ 
defined over $K$ and with Mordell-Weil group over $K$ of rank 
$mw_L(C) - \bigl( mw_L (B_L) - mw_K(B) \bigr)$.
\end{theorem}

\begin{proof}
The $L$-morphism $h \colon C \to B_L$ induces a $K$-morphism 
$\Wr(C) \to \Wr(B_L)$ which in turn induces a pull-back $K$-morphism 
$\Jac(\Wr(B_L)) \to \Jac(\Wr(C))$.  Furthermore, from the inclusion 
$\iota \colon \Res_h(C) \subset \Wr(C)$ we obtain a sequence of $K$-morphisms
\[
\Jac(\Wr(B_L)) \longrightarrow \Jac(\Wr(C)) 
\stackrel{\iota^*}{\longrightarrow} \Jac(\Res_h(C)) .
\]
From the representability of 
$\Wr(B_L)$, there is a $K$-morphism $\kappa \colon B \to \Wr(B_L)$ 
associated to the identity $B_L \to B_L$ using the bijection 
$\Hom_K(B,\Wr(B_L)) = \Hom_L(B_L,B_L)$.  The morphism $\kappa$ induces 
a pull-back $K$-morphism $\kappa^* \colon \Jac(\Wr(B_L)) \to \Jac(B)$; 
we denote by $M$ the kernel of the morphism of $\kappa^*$.  Geometrically, 
the Jacobian of $\Wr(B_L)$ is isomorphic to the product of $n$ copies 
of the Jacobian of $B$ and the morphism $\kappa^*$ corresponds to the 
addition of the line bundles in the various components using the 
isomorphisms between them (defined over an algebraic closure of $K$).  
We obtain that the group $M$ is the specialization to our setting of 
the group denoted also by $M$ in Theorem~\ref{thm:jaco}.  Therefore $M$ is 
geometrically isomorphic to $\Jac(B)^{n-1}$, and it is connected of 
dimension $(n-1) g(B)$.  Thus we obtain the diagram 
\[
\xymatrix{
M \ar@{^(->}[d] \ar[r] & \Jac(\Wr(C)) \ar[r]^{\iota^*} & \Jac(\Res_h(C)) \\
\Jac(\Wr(B_L)) \ar[d]^{\kappa^*} \ar[ur] \\
\Jac(B)
}
\]
of $K$-morphisms and it follows from Theorem~\ref{thm:jaco} that the 
group $M$ has finite index in $\ker(\iota^*)$.  
We let $F$ be the connected component of the identity of the image 
of $\iota^*$ and we show that it has the required properties.  First 
of all, $F$ is an abelian variety over $K$, isogenous over $K$ 
to $\Jac(\Wr(C))/M$, and hence the dimension of $F$ is 
\[
\dim (F) = \dim \left( \frac{\Jac(\Wr(C))}{M} \right) = n g(C) - (n-1) g(B) 
\]
as needed.  Next, we prove the statement about the Mordell-Weil rank 
of $F$.  For a curve $D$ defined over $L$ we have 
\begin{eqnarray*}
mw_K \bigl( \Wr(D) \bigr) & = 
& \rk \Bigl( \Jac \bigl( \Wr(D) \bigr) (K) \Bigr) \\
& = & \rk \biggl( \Hom_K \Bigl( \Spec(K) , 
\Jac \bigl( \Wr(D) \bigr) \Bigr) \biggr) \\
& = & \rk \biggl( \Hom_K \Bigl( \Spec(K) , 
\Wr \bigl( \Jac (D) \bigr) \Bigr) \biggr) \\
& = & \rk \Bigl( \Hom_L \bigl( \Spec(L) , \Jac (D) \bigr) \Bigr) \\
& = & mw_L(D) .
\end{eqnarray*}
Since the abelian varieties $F$ and $\Jac(\Wr(C))/M$ are $K$-isogenous, 
the ranks of their Mordell-Weil groups are the same.  By the previous 
computation we conclude that 
\begin{eqnarray*}
\rk \bigl( F(K) \bigr) & = & \rk 
\left( \frac{\Jac \bigl( \Wr(C) \bigr)}{M} (K) \right) \\
& = & mw_K \bigl( \Wr(C) \bigr) - \Bigl( mw_K \bigl( \Wr (B_L) \bigr) 
- mw_K(B) \Bigr) \\
& = & mw_L(C) - \bigl( mw_L (B_L) - mw_K(B) \bigr) ,
\end{eqnarray*}
as required, and the result follows.
\end{proof}

The theorem just proved opens the way to applications of the Chabauty 
method to find the $L$-rational points of the curve $C$ with $K$-rational 
image in $B$.

We show how this method works on an example.  

\begin{example}
Let $d$ be a squarefree integer; we let $K = \mathbb{Q}$ and 
$L = \mathbb{Q}(\sqrt{d})$.  Denote by $g(x) \in \mathbb{Q}[x]$ the polynomial 
\[
g(x) = x^3 + a x + b 
\]
and by $f(x) \in \mathbb{Q}(\sqrt{d})[x]$ the polynomial 
\[
f(x) = g(x^2+\sqrt{d}) . 
\]
Let $E$ be the elliptic curve over $\mathbb{Q}$ with Weierstrass 
equation $y^2=g(x)$ and let $C$ be the smooth projective model over 
$\mathbb{Q}(\sqrt{d})$ of the genus two hyperelliptic curve with 
affine equation $y^2=f(x)$.  By construction, there is a morphism 
$\phi \colon C \to E$ given by 
\[
\begin{array}{rcl}
\phi \colon C & \longrightarrow & E
\\[5pt]
(x,y) & \longmapsto & ( x^2 + \sqrt{d} , y ).
\end{array}
\]
Suppose we wish to find all points $P$ in 
$C\bigl( \mathbb{Q}(\sqrt{d}) \bigr) $ such that $\phi(P)$ is in 
$E(\mathbb{Q})$.  Such points are the rational points of the curve 
$D = \Res_\phi(C)$ over $\mathbb{Q}$, for which we now determine an 
explicit model.  Let $x =  x_1 +  x_2 \sqrt{d}$ and $y = y_1 + y_2 \sqrt{d}$, 
where $ x_1, x_2,y_1,y_2$ are $\mathbb{Q}$-rational variables. 
Substituting $x = x_1+x_2\sqrt{d}$ and $y = y_1+y_2\sqrt{d}$ in the 
polynomial defining $C$ we find the polynomial 
\[
r(x_1,x_2,y_1,y_2) = (y_1 + y_2 \sqrt{d})^2 - f(x_1 +  x_2 \sqrt{d})
\]
in $\mathbb{Q}(\sqrt{d})$, whose vanishing represents the condition 
that the point $P = ( x_1 +  x_2 \sqrt{d}, y_1 + y_2 \sqrt{d})$ lies 
on $C$.  Define polynomials with rational coefficients 
\[
r_1 := \frac{r + \overline{r}}{2}
\quad {\textrm{and}} \quad 
r_2 := \frac{r - \overline{r}}{2 \sqrt{d}}
\]
where $\overline{r}$ denotes the polynomial obtained from $r$ by applying 
the nontrivial element of the Galois group of 
$\mathbb{Q}(\sqrt{d})/\mathbb{Q}$; we have the identity 
$r = r_1 + \sqrt{d} r_2$.  Thus, the two equations 
$r_1=r_2=0$ in $ x_1,y_1, x_2,y_2$ correspond to $P$ lying on $C$; 
these are also the equations defining (an affine model of) the Weil 
restriction of $C$ from $\mathbb{Q}(\sqrt{d})$ to $\mathbb{Q}$.  To 
determine $D = \Res_\phi(C)$, we also wish $\phi(P)$ to be in 
$E(\mathbb{Q})$. Note that the coordinates of $\phi(P)$ are 
\[
\phi(P) = \bigl( x_1^2 + d x_2^2 
+ \sqrt{d} (2 x_1 x_2 + 1) , y_1+\sqrt{d} y_2 \bigr)
\]
and the condition that the point $\phi(P)$ lies in $\mathbb{Q}$ translates 
to the equations $2x_1 x_2 + 1 = y_2 = 0$.  We have therefore obtained 
the four equations 
\begin{equation} \label{rima}
\left\{
\begin{array}{rcl}
y_1^2 & = & (x_1^6 + d^3 x_2^6) + 3(x_1^2 + d x_2^2) 
\bigl( 5 d x_1^2 x_2^2 + 4 d x_1 x_2 + d + \frac{a}{3} \bigr) + b
\\[7pt]
0 & = & (2 x_1 x_2+1) (3  x_1^4+10  x_1^2  x_2^2 d+4  x_1  x_2 d
+3  x_2^4 d^2+a+d) 
\\[7pt]
0 & = & 2 x_1 x_2+1 
\\[5pt]
y_2 & = & 0 
\end{array}
\right.
\end{equation}
in the variables $ x_1,y_1, x_2,y_2$.  But of course the second equation is 
divisible by the third equation, and we may ignore it (this is not a 
coincidence, but it is a consequence of the fact that the curve $E$ is 
defined over $\mathbb{Q}$). 
Multiplying the first equation in~\eqref{rima} by $2^{12}x_1^6$, we can 
use the relation $2  x_1  x_2=-1$ to eliminate the variable $x_2$, obtaining 
the single equation $(2^6x_1^3 y_1)^2 = \overline{\rho} (x_1)$ in $x_1, y_1$ 
for $D$.  After the birational substitution $x=2x_1$ and $y=2^6x_1^3 y_1$, 
we obtain that the curve $D$ is birational to the genus 5 curve with equation
\[
D \colon y^2 = x^{12} + 4(3 d + 4 a) x^8 + 64 b x^6 
+ 16d (3 d + 4 a) x^4 + 64 d^3 .
\]
The curve $D$ admits the non-constant map $(x,y) \mapsto (x^2,y)$ 
to the genus 2 curve $F$ with equation 
\[
F \colon y^2 = x^6 + 4(3 d + 4 a) x^4 + 64 b x^3 + 16d (3 d + 4 a) x^2 
+ 64 d^3.
\]
Summarizing, the curve $D$ is a genus 5 curve defined over $\mathbb{Q}$ 
and it admits two morphisms defined over $\mathbb{Q}(\sqrt{d})$ to the curves 
\[
C \colon y^2=f(x) = g(x^2+\sqrt{d}) \quad {\textrm{ and }} 
\quad \overline{C} \colon y^2=\overline{f}(x) = g(x^2-\sqrt{d}) .
\]
Each of the two curves $C$ and $\overline{C}$ admits a 
$\mathbb{Q}(\sqrt{d})$-morphism to the elliptic curve $E$, and the 
corresponding two compositions $D \to E$ coincide, and are defined 
over $\mathbb{Q}$.  We deduce that the Jacobians of $C$ and of 
$\overline{C}$ are contained, up to isogeny, in the Jacobian of $D$, 
and they have an isogenous copy of $E$ in common.  This implies that the 
five-dimensional Jacobian of $D$ contains a further two-dimensional 
abelian variety: up to isogeny this is the Jacobian of the curve $F$.

It is now easy to provide examples where the inequality of the Chabauty 
method is not satisfied for the curve $C$, nor for the curve $E$, but it 
satisfied for the curve $F$, so that we can still apply the Chabauty 
method to find the points on $D$.
For example, setting $a = 1$, $b = 3$, and $d = 13$, we find 
$\rk \bigl( E(\mathbb{Q}) \bigr) = 1$ and 
$\rk \bigl( \Jac(F)(\mathbb{Q}) \bigr) = 1$, and moreover 
\[
\rk \left( \Jac(C) \bigl( \mathbb{Q}(\sqrt{d}) \bigr) \right) 
\geq \rk \left( E \bigl( \mathbb{Q}(\sqrt{d}) \bigr) \right) \geq 2 .
\]
Thus, in this example the Chabauty method is not applicable to $C$ or $E$, 
but it is only applicable to $F(\mathbb{Q})$.
\end{example}

\section{Cases where Faltings' Theorem does not apply}

In this section we analyze the cases where the relative Weil restriction 
of a morphism of curves contains a component of geometric genus at 
most one.  In such cases, Faltings' Theorem cannot be applied to deduce 
the finiteness of rational points of the relative Weil restriction and we 
find explicit non-tautological examples in which these sets of rational 
points are infinite.  The following remark is an immediate consequence 
of Faltings' Theorem and guides the choice of cases we handle in this section.

\begin{remark} \label{gamu}
Let $L/K$ be an extension of number fields, and suppose that $C$ is a curve 
defined over $L$, $B$ is a curve defined over $K$ and $h \colon C \to B$ 
is a non-constant morphism.  If the curve $\Res_h(C)$ has infinitely many 
$K$-rational points, then the genus of $C$ is at most one.
\end{remark}

In the case of the relative Weil restriction of a morphism from a curve 
of genus one, we completely characterize the cases where the set of 
rational points is not finite.

\begin{proposition} \label{p:quo}
Let $C , X_1 , \ldots , X_n , B$ be smooth projective curves (not necessarily 
connected), and let 
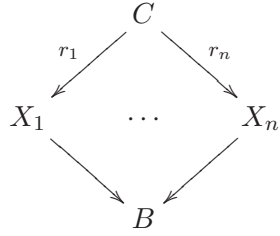
\[
\xymatrix{
 & C \ar[dl] _{r_1} 
 \ar[dr] ^{r_n} 
 \\ 
X_1 \ar[dr] & \cdots & X_n \ar[dl] \\
 & B 
}
\]
be a commutative diagram, where all morphisms are finite and flat.  There 
is a smooth projective curve $X_C$ and a commutative diagram 
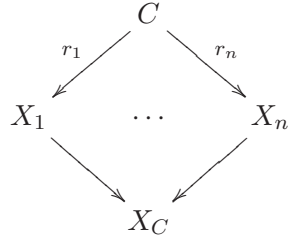
\[
\xymatrix{
 & C \ar[dl] _{r_1} 
 \ar[dr] ^{r_n} 
 \\ 
X_1 \ar[dr] & \cdots & X_n \ar[dl] \\
 & X_C 
}
\]
such that for every projective curve $D$ fitting in the commutative diagram 
of solid arrows
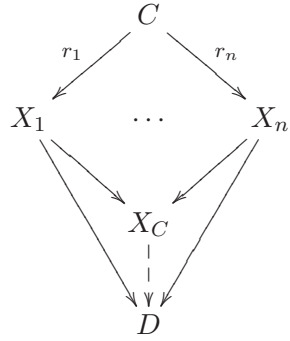
\[
\xymatrix{
 & C \ar[dl] _{r_1} 
 \ar[dr] ^{r_n} 
 \\ 
X_1 \ar[dr] \ar[ddr] & \cdots & X_n \ar[dl] \ar[ddl] \\
 & X_C \ar@{-->}[d] \\
 & D 
}
\]
the dashed arrow $\xymatrix{X_C \ar@{-->}[r] & D}$ exists uniquely.
\end{proposition}

\begin{proof}
Let $X$ be the disjoint union $X := X_1 \sqcup \cdots \sqcup X_n$.  The 
morphisms $r_1,\ldots,r_n$ determine an equivalence relation $\sim_C$ on $X$, 
where $x \sim_C y$ if there is a sequence $(i_1 , \ldots , i_t)$ of elements 
of $\{1,\ldots,n\}$, and points $x_1 = x \in X_{i_1}$, 
$x_2 \in X_{i_2}$, \ldots , $x_{t-1} \in X_{i_{t-1}}$, $x_t = y \in X_{i_t}$ 
and a sequence $c_1 , \ldots , c_{t-1}$ of points of $C$ such that for 
all $j \in \{1,\ldots,t-1\}$ we have $r_{i_j} (c_i) = r_{i_{j+1}} (c_i)$.  
Thus, a pair $(x,y) \in X \times X$ is in the relation determined by 
the morphisms $r_1,\ldots,r_n$ if and only if there is a sequence 
$I := (i_1 , \ldots , i_t)$ of elements of $\{1,\ldots,n\}$ such that 
the pair $(x,y)$ is in the image of the composition 
\[
C \times _{X_{i_2}} \times \cdots \times _{X_{i_{t-1}}} C 
\stackrel{\pi}{\longrightarrow} C \times C 
\stackrel{(r_{i_1},r_{i_t})}{\longrightarrow} X_{i_1} 
\times X_{i_t} \subset X \times X
\]
where $\pi$ is the projection to the first and last factor; denote 
by $C_I \subset X \times X$ the image of this morphism.  Observe that, 
for every finite sequence $I$ of elements of $\{1 , \ldots , n\}$, the 
scheme $C_I$ is a closed subscheme of $X \times X$ that is finite and 
flat over each factor $X$ and hence also over $B$; in particular, each 
scheme $C_I$ has non-zero degree over $B$.  Moreover, if the pair $(x,y)$ 
is in the relation $\sim_C$, then $x$ and $y$ have the same image in $B$.  
From this it follows that pairs $(x,y)$ in the relation $\sim_C$ are 
covered by at most $\bigl( \sum _i \deg(f_i) \bigr)^2$ of the schemes $C_I$ 
defined above since they are contained in $X \times _B X$.  It follows 
that the whole graph $R_C \subset X \times X$ of the relation $\sim_C$ 
is a subscheme of finite type of $X \times X$ that is flat and proper 
over each factor.  The hypotheses of~\cite{SGA3}, Th\'eor\`eme~V.7.1,
are therefore satisfied and the result follows.
\end{proof}

\begin{corollary} \label{coge1}
With the notation of the previous proposition, assume further that the 
curves $C,X_1 , \ldots , X_n$ all have genus one.  Then the curve $X_C$ 
is a torsor under $\Jac (C)/K$, where $K$ is the subgroup of $\Jac (C)$ 
generated by the kernels of the morphisms 
$\Jac (C) \to \Jac (X_1)$, \ldots , $\Jac (C) \to \Jac (X_n)$.
\end{corollary}

\begin{proof}
We can clearly assume that the ground field is algebraically closed, and 
further reduce to the case in which the morphisms $r_1, \ldots , r_n$ 
are all homomorphisms of elliptic curves.  Thus $K \subset C$ is identified 
with the subgroup generated by the kernels of all the morphisms 
$r_1, \ldots , r_n$ and let $C' := C/K$ denote the quotient of $C$ by 
the subgroup $K$.  Clearly, the curves $X_1 , \ldots , X_n$ all admit 
a morphism to $C'$ making the diagram of Proposition~\ref{p:quo} commute.  
It follows from the previous proposition that $X_C$ admits a morphism 
to $C'$ that is necessarily non-constant, and we conclude that $X_C$ 
has genus one.  Moreover, it is also clear that the curve $X_C$ is 
isomorphic to the curve $C'$: if $D$ is any curve making the diagram 
of Proposition~\ref{p:quo} commute, then the fiber in $C$ of the 
morphism $C \to D$ over the image of the origin in $C$ contains all 
the kernels of the morphisms $r_1 , \ldots , r_n$, and hence it contains 
the subgroup $K$, since the morphism factors through the elliptic curve $X_C$.
\end{proof}

Thus we see that if the curves $X_1,\ldots,X_n$ have genus one and if 
the fibered product $X_1 \times _B \cdots \times _B X_n$ contains a 
geometrically integral curve of geometric genus one defined over the 
ground field, then the morphism $X_1 \to B$ factors through a 
morphism $E \to B$ defined over the ground field, where $E$ is a 
smooth geometrically integral curve of geometric genus one and the 
morphism $X_1 \to E$ is an isogeny defined over the extension field.

\subsection{Genus one}
We specialize what we just proved to the case of the relative Weil restriction 
from an elliptic curve.  Let $L/K$ be an extension of number fields; let $E$ 
be an elliptic curve defined $L$ and let $B$ be a smooth projective integral 
curve defined over $K$.  Suppose that $h \colon E \to B_L$ is a non-constant 
morphism.  The relative Weil restriction $\Res_h(E)$ is a curve defined 
over $K$.  Fix an algebraic closure $\overline{K}$ of $K$, denote by 
$\sigma_1 , \ldots , \sigma_n$ the distinct embeddings of $L/K$ in 
$\overline{K}$, and let $E_1 , \ldots , E_n$ be the corresponding Galois 
conjugates of $E$.  Over the field $\overline{K}$, the curve $\Res_h(E)$ 
is isomorphic to $E_1 \times _B \cdots \times_B E_n$ (Lemma~\ref{profi}).  
Suppose that the curve $\Res_h(E)$ contains infinitely many $K$-rational 
points.  It follows from Faltings' Theorem that there is a component $C$ 
of $\Res_h(E)$ of geometric genus at most one, defined over $K$; if $C$ 
is not normal, we replace it by its normalization.  Since the morphism 
$\Res_h(E) \to B$ is flat, all its fibers are finite and therefore the 
curve $C$ is also finite over $B$.  In particular, the $L$-morphisms 
$C \to E_1 , \ldots , C \to E_n$ are all finite, since all the curves 
are smooth.  Let $E_C$ denote the universal curve fitting in the diagram 
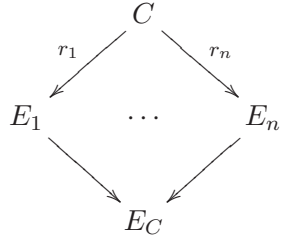
\[
\xymatrix{
 & C \ar[dl] _{r_1} 
 \ar[dr] ^{r_n} 
 \\ 
E_1 \ar[dr] & \cdots & E_n \ar[dl] \\
 & E_C 
}
\]
of Proposition~\ref{p:quo}.  Since the curves $C,E_1,\ldots,E_n$ all 
have genus one, we may therefore apply Corollary~\ref{coge1} to deduce 
that the curve $E_C$ is also a torsor under an elliptic curve, and 
therefore also $E_C$ has genus one.  In the case in which $L$ is a 
number field, we obtain the following corollary.

\begin{corollary}
Let $L/K$ be an extension of number fields and suppose that $E$ is an 
elliptic curve over $L$, $B$ is a curve over $K$, and $h \colon E \to B_L$ 
is a non-constant morphism.  If the set of $K$-rational points of 
$\Res_h(E)$ is infinite, then the curve $E$ is $L$-isogenous to an 
elliptic curve defined over $K$ having positive rank over $K$.
\end{corollary}

\begin{proof}
By Faltings' Theorem we deduce that $\Res_h(E)$ contains a geometrically 
integral component $E'$ of genus at most one defined over $K$ and having 
infinitely many $K$-rational points.  Since $E'$ admits a non-constant 
$L$-morphism to $E$, it follows that $E'$ has genus one and that it 
is $L$-isogenous to $E$, as required.
\end{proof}

\begin{remark}
In the case in which $E$ and $B$ have genus one, then all the geometric 
components of $\Res_h(E)$ have genus one.  Hence, finding the $K$-rational 
points of $\Res_h(E)$ is equivalent to finding the $K$-rational points 
of finitely many elliptic curves that are $K$-isogenous to $B$.
\end{remark}

This completes our analysis in the case in which the curve $C$ of 
Remark~\ref{gamu} has genus one.  We next discuss the case in which $C$ 
has genus zero.  We are not able to give a treatment of this case that 
is as detailed as the case of genus one.

\subsection{Genus zero}
We specialize to the case in which the curve $C$ is isomorphic to 
$\mathbb{P}^1_L$ and hence $B$ is isomorphic to $\mathbb{P}^1_K$.  
The morphism $h \colon \mathbb{P}^1_L \to \mathbb{P}^1_L = (\mathbb{P}^1_K)_L$ 
is therefore determined by a rational function $F \in L(x)$.  The set 
of $K$-rational points of $\Res_h(\mathbb{P}^1_L)$ is essentially 
the set $A_F \subset L$ of values of $x \in L$ such that $F(x)$ lies 
in $K$, mentioned in the introduction.  

\begin{definition}
Let $G \colon C \to D$ be a finite morphism between smooth curves, 
let $p \in D$ be a geometric point.  The {\em{type}} of $p$ is the 
partition $\lambda_p$ of $\deg(G)$ determined by the fiber $G^{-1}(p)$.  
We extend this definition to the case in which the curve $C$ is reduced, 
but not necessarily smooth, by replacing $G$ with the morphism 
$G^\nu \colon C^\nu \to D$, where $C^\nu$ is the normalization of $C$ 
and $G^\nu$ is the morphism induced by $G$.
\end{definition}

In this section we restrict our attention to field extensions 
$L \supset K$ of degree at most three.

\subsubsection*{Degree two}
Let $L \supset K$ be a field extension of degree two, let 
$F \colon \mathbb{P}^1_L \to \mathbb{P}^1_L$ be a morphism of degree 
three.  We are interested in the values $\ell \in L$ such 
that $f(\ell) \in K$.  Assume that $L = K(\sqrt{d})$ where 
$d \in K \setminus K^2$; write $\alpha = a + b \sqrt{d}$, for $a,b \in K$.  
Parameterizing the rational curve $\Res_F(F)$ we find that, for every 
element $t$ in $K$, the evaluation 
\[
p \left( \frac{d b t^2+2 a t+b}{t (d t^2 + 3)} (\sqrt{d} t + 1) \right)
\]
is also in $K$.

\subsubsection*{Degree three}
Suppose that $F \colon \mathbb{P}^1_L \to \mathbb{P}^1_L$ is a morphism 
of degree three, and suppose that the characteristic of $K$ is neither 
two nor three.  As usual, we are interested in the values $\ell \in L$ 
such that $f(\ell) \in K$.  Denote by $\overline{F}$ the morphism 
conjugate to $F$ under the Galois involution of $L/K$.  We construct 
examples of morphisms $F$ such that $\Res_F(\mathbb{P}^1_L)$ is a 
geometrically integral curve of geometric genus zero.  Note that the 
curve $\Res_F(\mathbb{P}^1_L)$ has a line bundle of degree nine given 
by pull-back of $\mathcal{O}_{\mathbb{P}^1_K}(1)$; since this curve is 
geometrically rational and it has a line bundle of odd degree, it follows 
that it is rational over $K$.

Denote by $\Res_F(F)^\nu$ the composition of the normalization map 
of $\Res_F(\mathbb{P}^1_L)$ and $\Res_F(F)$.  Applying the Hurwitz formula 
to the morphisms $F$, $\overline{F}$ and $\Res_F(F)^\nu$, we find that the 
respective total degrees of the ramification divisors are $4$, $4$ and $16$.

We begin by analyzing the ramification patterns.  For a geometric point 
$p \in \mathbb{P}^1$, Table~\ref{fihu} shows the possibilities of the 
types of the fibers of the three morphisms $F$, $\overline{F}$ 
and $\Res_F(F)^\nu$ and the contributions of each to the Hurwitz formula.  
In our setup, the Galois involution of $L/K$ induces a bijection between 
fiber types of $F$ and of $\overline{F}$: this is recorded in the last 
column of Table~\ref{fihu}.
\begin{table} \label{fihu}
\begin{tabular}{|c|c|c|c|}
\hline
Fiber type of & Fiber type of & Hurwitz contribution to & Symmetrized \\
the pair $F,\overline{F}$ & $\Res_F(F)$ & $[F,\overline{F}]$ and 
$[\Res_F(F)]$ & contribution \\[5pt]
\hline
$\bigl( (1,1,1) , (1,1,1) \bigr)$ & $(1,1,1 , 1,1,1 , 1,1,1)$ 
& [0,0] , [0] & [0,0] , [0] \\[5pt]
$\bigl( (2,1) , (1,1,1) \bigr)$ & $(2,2,2 , 1,1,1)$ 
& [1,0] , [3] & [1,1] , [6] \\[5pt]
$\bigl( (3) , (1,1,1) \bigr)$ & $(3,3,3)$ & [2,0] , [6] & [2,2] , [12] \\[5pt]
$\bigl( (2,1) , (2,1) \bigr)$ & $(2,2,2,2,1)$ 
& [1,1] , [4] & [1,1] , [4] \\[5pt]
$\bigl( (2,1) , (3) \bigr)$ & $(6,3)$ & [1,2] , [7] & [3,3] , [14] \\[5pt]
$\bigl( (3) , (3) \bigr)$ & $(3,3,3)$ & [2,2] , [6] & [2,2] , [6] \\[5pt]
\hline
\end{tabular}
\caption{Fiber types and contributions to the Hurwitz formula for 
fiber products of morphisms of degree three}
\end{table}

\noindent
It is now easy to check that the only possibilities for the fiber 
types of the morphisms $F,\overline{F}$ are 
\begin{equation}
\bigl( (3) , (3) \bigr) + \bigl( (2,1) , (1,1,1) \bigr) 
+ \bigl( (1,1,1) , (2,1) \bigr) + \bigl( (2,1) , (2,1) \bigr) \label{pollo}
\end{equation}
and
\begin{equation}
\bigl( (2,1) , (2,1) \bigr) + \bigl( (2,1) , (2,1) \bigr) 
+ \bigl( (2,1) , (2,1) \bigr) + \bigl( (2,1) , (2,1) \bigr) . \label{simme}
\end{equation}
We only analyze the case~\eqref{pollo}.

\subsubsection*{Fiber types~\eqref{pollo}}
The fiber type $((3),(3))$ in~\eqref{pollo} implies that the coordinate 
on $\mathbb{P}^1$ can be chosen so that the morphism $F$ is a polynomial 
and the fiber type $((2,1) , (2,1))$ shows that one of the ramification 
points in defined over $K$.  This case is realized by morphisms 
$F \colon \mathbb{P}^1_L \to \mathbb{P}^1_L$ of the form 
$p(x) = x^2(x-\alpha)$, for $\alpha \in L$.  Since the derivative 
of $p$ vanishes at $0$ and at $\frac{2\alpha}{3}$, it follows that 
the ramification types of $F,\overline{F}$ are of the form~\eqref{pollo} 
when $\alpha \notin K$.

\end{document}